\documentclass[12pt, reqno]{amsart}
\usepackage{ amsmath,amsthm, amscd, amsfonts, amssymb, graphicx, color}
\usepackage[bookmarksnumbered, colorlinks, plainpages]{hyperref}
\newtheorem{thm}{Theorem}[section]
\newtheorem{cor}[thm]{Corollary}
\newtheorem{lem}[thm]{Lemma}

\newtheorem{exam}[thm]{Example}
\numberwithin{equation}{section}

\begin{document}

\title{the group inverse of the sum in a Banach algebra}

\author{Huanyin Chen}
\author{Marjan Sheibani$^*$}
\address{
Department of Mathematics\\ Hangzhou Normal University\\ Hang -zhou, China}
\email{<huanyinchen@aliyun.com>}
\address{Women's University of Semnan (Farzanegan), Semnan, Iran}
\email{<sheibani@fgusem.ac.ir>}

 \thanks{$^*$Corresponding author}


\begin{abstract} In this paper, we present new necessary and sufficient conditions under which the sum of two group invertible elements in a Banach algebra has group inverse. We then apply these results to block operator matrices over Banach spaces. The group inverses of certain operator block matrices are thereby obtained. Additionally, this paper extends the results obtained in ~\cite{L}.\end{abstract}

\maketitle

\section{Introduction}

Let $\mathcal{A}$ be a complex Banach algebra with an identity $1$. An element $a$ in a Banach algebra $\mathcal{A}$ has group inverse provided that there exists $b\in \mathcal{A}$ such that $ab=ba, b=bab$ and $a=aba$. Such $b$ is unique if exists, denoted by $a^{\#}$, and called the group inverse of $a$. As is well known, a square complex matrix $A$ has group inverse if and only if $rank(A)=rank(A^2)$.

The group invertibility in a Banach algebra is attractive. Many authors have studied such problems from many different views, e.g., ~\cite{B,L,LD,M,ZCZ}. It was also extensively investigated under the concept "strongly regularity" in ring theory.

In Section 2, we present necessary and sufficient conditions under which the sum of two group invertible elements in a Banach algebra. Let $a,b\in \mathcal{A}^{\#}$. If $ab=\lambda ba$ for some $\lambda\in {\Bbb C}$, then $a+b\in \mathcal{A}^{\#}$ if and only if $abb^{\#}+baa^{\#}\in \mathcal{A}^{\#}$. In this case, $(a+b)^{\#}=(abb^{\#}+baa^{\#})^{\#}+a^{\#}b^{\pi}+b^{\#}a^{\pi}.$

Let $X$ and $Y$ be Banach spaces. In Section 3, we investigate the group inverse of a $2\times 2$ operator matrix $$M=\left(
\begin{array}{cc}
A&B\\
C&D
\end{array}
\right)~~~~~~~~~~(*)$$  where $A\in \mathcal{L}(X), D\in \mathcal{L}(Y)$. Here, $M$ is a bounded linear operator on $X\oplus Y$. This problem is quite complicated and was expensively studied by many authors. Let $A,D, BC$ have group inverses. If $AB=0, BD=0, B(CB)^{\pi}=0, C(BC)^{\pi}=0$ and $DC=\lambda CA$, we prove that $M$ has group inverse. The explicit formula of the group inverse of $M$ is also given.

If $a\in \mathcal{A}$ has the group inverse $a^{d}$, the element $a^{\pi}=1-aa^{d}$ is called the spectral idempotent of $a$. Let $p\in \mathcal{A}$ be an idempotent, and let $x\in \mathcal{A}$. Then we write $$x=pxp+px(1-p)+(1-p)xp+(1-p)x(1-p),$$
and induce a Pierce representation given by the matrix
$$x=\left(\begin{array}{cc}
pxp&px(1-p)\\
(1-p)xp&(1-p)x(1-p)
\end{array}
\right)_p.$$
Throughout the paper, all Banach algebras are complex with an identity. We use $\mathcal{A}^{\#}$ to denote the set of all group invertible elements in $\mathcal{A}$. $\lambda$ stands for a nonozero complex numbers.

\section{additive properties}

The purpose of this section is to establish the new formulas for the group  of the sum $a+b$ which will be used in the sequel. We begin with

\begin{lem} Let $\mathcal{A}$ be a Banach algebra, and let $a,b\in\mathcal{A}^{\#}$. If $ab=\lambda ba$, then the following hold:\begin{enumerate}
\item [(1)] $abb^{\#}=bb^{\#}a$.
\item [(2)] $baa^{\#}=aa^{\#}b$.
\end{enumerate}
\end{lem}
\begin{proof} $(1)$ As $b$ has group inverse, $b^{\#}$ coincides with its Drazin inverse, and so it follows by~\cite[Theorem 2.3]{C1} that
$ab^{\#}=\frac{1}{\lambda}b^{\#}a$. Therefore $$\begin{array}{lll}
abb^{\#}&=&\lambda (ba)b^{\#}\\
&=&\lambda b(\frac{1}{\lambda}b^{\#}a)\\
&=&bb^{\#}a.
\end{array}$$

$(2)$ This is similar to $(1)$.\end{proof}

\begin{thm} Let $a,b\in \mathcal{A}^{\#}$. If $ab=\lambda ba$, then the following are equivalent:\begin{enumerate}
\item [(1)] $a+b\in \mathcal{A}^{\#}$.
\item [(2)] $a(1+a^{\#}b)\in \mathcal{A}^{\#}$.
\item [(3)] $abb^{\#}+baa^{\#}\in \mathcal{A}^{\#}$.\end{enumerate}
In this case, $$\begin{array}{lll}
(a+b)^{\#}&=&a^{\#}b^{\pi}+\big[a(1+a^{\#}b)\big]^{\#}\\
&=&a^{\#}b^{\pi}+a^{\pi}b^{\#}+\big[abb^{\#}+baa^{\#}\big]^{\#}.
\end{array}$$\end{thm}
\begin{proof} $(1)\Leftrightarrow (2)$ Let $p=aa^{\#}$. Then we have
$$a=\left(\begin{array}{cc}
a_1&0\\
0&0
\end{array}
\right)_p, b=\left(\begin{array}{cc}
b_{1}&b_{12}\\
b_{21}&b_2
\end{array}
\right)_p.$$ Since $ab=\lambda ba,$ we get $b_{12}=b_{21}=0$. Thus,
$$a=\left(\begin{array}{cc}
a_1&0\\
0&0
\end{array}
\right)_p, b=\left(\begin{array}{cc}
b_1&0\\
0&b_2
\end{array}
\right)_p,$$
Clearly, $a_1=aa^{\#}aaa^{\#}$, $b_1=aa^{\#}baa^{\#}$ and $b_2=a^{\pi}ba^{\pi}$.
Hence,
$$a+b=\left(
\begin{array}{cc}
aa^{\#}(a+b)aa^{\#}&0\\
0&a^{\pi}ba^{\pi}
\end{array}
\right).$$
In view of Lemma 2.1, $a^{\pi}b=ba^{\pi}$. Then we easily check that $a^{\pi}ba^{\pi}$ has group inverse.
Therefore $a+b$ has group inverse if and only if $aa^{\#}(a+b)aa^{\#}$ has group inverse. In this case,
$$(a+b)^{\#}=a^{\#}b^{\pi}+\big[a(1+a^{\#}b)\big]^{\#}.$$

$(2)\Leftrightarrow (3)$ Write $aa^{\#}(a+b)aa^{\#}=a+aa^{\#}b:=a_1+b_1$. Choose $q=b_1b_1^{\#}$. Then
Then we have
$$a_1=\left(\begin{array}{cc}
a_1'&0\\
0&a_2'
\end{array}
\right)_p, b_1=\left(\begin{array}{cc}
b_{1}'&0\\
0&0
\end{array}
\right)_q,$$ where $a_1=b_1b_1^{\#}ab_1b_1^{\#}, a_2'=b_1^{\pi}ab_1^{\pi}.$
Hence,
$$a_1+b_1=\left(
\begin{array}{cc}
aa^{\#}(a+b)bb^{\#}&0\\
0&b_1^{\pi}ab_1^{\pi}
\end{array}
\right).$$
In view of Lemma 2.1, $b_1^{\pi}ab_1^{\pi}$ has group inverse.
Therefore $a_1+b_1$ has group inverse if and only if $aa^{\#}(a+b)bb^{\#}$ has group inverse. In this case,
$$\begin{array}{lll}
(a_1+b_1)^{\#}&=&a^{\#}b^{\pi}+[aa^{\#}(a+b)bb^{\#}]^{\#}\\
&=&a^{\#}b^{\pi}+[aa^{\#}(a+b)bb^{\#}]^{\#}.
\end{array}$$
Combing the preceding formulas, we derive that
$$(a+b)^{\#}=a^{\#}b^{\pi}+a^{\pi}b^{\#}+[abb^{\#}+baa^{\#}]^{\#},$$ as asserted.\end{proof}

\begin{cor} Let $\mathcal{A}$ be a Banach algebra, and let $a,b\in \mathcal{A}^{\#}$. If $ab=\lambda ba,$ then the following are equivalent:\end{cor}
\begin{enumerate}
\item [(1)]{\it $a-b\in \mathcal{A}^{\#}.$}
\vspace{-.5mm}
\item [(2)]{\it $a(1-a^{\#}b)\in \mathcal{A}^{\#}.$}
\end{enumerate} In this case, $$(a-b)^{\#}=a^{\#}b^{\pi}-a^{\pi}b^{\#}+[a(1-a^{\#}b)]^{\#}.$$
\begin{proof} Clearly, $-b\in \mathcal{A}^{\#}.$ Applying Theorem 2.2 to $a$ and $-b$, we obtain the result.\end{proof}

\begin{cor} Let $a,b\in \mathcal{A}^{\#}.$ If $ab=ba$, then the following are equivalent: \end{cor}
\begin{enumerate}
\item [(1)]{\it $a+b\in \mathcal{A}^{\#}$.}
\vspace{-.5mm}
\item [(2)]{\it $1+a^{\#}b\in \mathcal{A}^{\#}$.}
\end{enumerate}
In this case, $$(a+b)^{\#}=a^{\#}(1+a^{\#}b)^{\#}bb^{\#}+a^{\#}b^{\pi}+b^{\#}a^{\pi}.$$
\begin{proof} $(1)\Rightarrow (2)$ We check that $$1+a^{\#}b=a^{\pi}+a^{\#}(a+b).$$  Since $a^{\pi}a^{\#}(a+b)=0=a^{\#}(a+b)a^{\pi}$, it follows by~\cite[Theorem 2.1]{B} that $1+a^{\#}b\in \mathcal{A}^{\#}$.

$(2)\Rightarrow (1)$ In view of Lemma 2.1, we see that $$abb^{\#}+baa^{\#}=abb^{\#}(1+a^{\#}b).$$
Clearly, $abb^{\#}\in \mathcal{A}^{\#}$. Since $abb^{\#}, 1+a^{\#}b\in  \mathcal{A}^{\#}$ and $abb^{\#}(1+a^{\#}b)=(1+a^{\#}b)abb^{\#}$, it follows that
$abb^{\#}(1+a^{\#}b)\in \mathcal{A}^{\#}$. Hence $abb^{\#}+baa^{\#}\in \mathcal{A}^{\#}$.
Moreover, $$(abb^{\#}+baa^{\#})^{\#}=[abb^{\#}(1+a^{\#}b)]^{\#}=a^{\#}(1+a^{\#}b)^{\#}bb^{\#}.$$
According to Theorem 2.2, we complete the proof.\end{proof}

We come now to extend ~\cite[Corollary 3.3]{L} to the group inverse of an element in Banach algebra.

\begin{cor} Let $a,b\in \mathcal{A}^{\#}$ and $\lambda,\mu\in {\Bbb C}\setminus \{0\}$. If $ab=ba$, then the following are equivalent: \end{cor}
\begin{enumerate}
\item [(1)]{\it $\lambda a+\mu b\in \mathcal{A}^{\#}$.}
\vspace{-.5mm}
\item [(2)]{\it $\lambda abb^{\#}+\mu baa^{\#}\in \mathcal{A}^{\#}$.}
\end{enumerate}
In this case, $$(\lambda a+\mu b)^{\#}=(\lambda abb^{\#}+\mu baa^{\#})^{\#}+\frac{1}{\lambda}ab^{\#}+\frac{1}{\mu}b^{\#}a^{\pi}.$$
\begin{proof} Since $ab=ba$, we have $(\lambda a)(\mu b)=(\mu b)(\lambda a)$. Clearly, $(\lambda a)^{\#}=\frac{1}{\lambda}a^{\#}, (\mu b)^{\#}=\frac{1}{\mu}b^{\#}$. Applying Theorem 2.2 to $\lambda a, \mu b$, we obtain the result.\end{proof}

\begin{exam} Let $a=\left(
\begin{array}{cc}
0&1\\
1&0
\end{array}
\right), b=\left(
\begin{array}{cc}
-1&0\\
0&1
\end{array}
\right)\in {\Bbb C}^{2\times 2}$. Then $ab=-ba$ and $$a^{\#}=\left(
\begin{array}{cc}
0&1\\
1&0
\end{array}
\right), b^{\#}=\left(
\begin{array}{cc}
-1&0\\
0&1
\end{array}
\right).$$ In view of Theorem 2.2, $a+b$ has group inverse. Evidently,
$$(a+b)^{\#}=\left(
\begin{array}{cc}
-\frac{1}{2}&\frac{1}{2}\\
\frac{1}{2}&\frac{1}{2}
\end{array}
\right).$$ In this case, $ab\neq ba$.\end{exam}

\section{Block operator matrices}

The aim of this section is to present certain simpler representations of the group inverse of the block matrix $M$ by applying Theorem 2.2.

\begin{lem} Let $A$ and $D$ have group inverses. If $D^{\pi}C=0$, then $\left(
  \begin{array}{cc}
    A & 0 \\
    C & D
  \end{array}
\right)$ has group inverse and
$$\begin{array}{l}
\left(
  \begin{array}{cc}
    A & 0 \\
    C & D
  \end{array}
\right)^{\#}=\left(
\begin{array}{cc}
A^{\#}&0\\
-D^{\#}CA^{\#}+(D^{\#})^2CA^{\pi}&D^{\#}
\end{array}
\right).
\end{array}$$\end{lem}
\begin{proof} This is obvious by ~\cite[Theorem 3.4]{B}.\end{proof}

Recall that $a\in \mathcal{A}$ is called Drazin invertible provided that there is a common solution to the
equations $ab=ba, b=bab$ and $a^k=a^kba$ for som $k\geq 1$. If such a solution exists, then it is unique and is called a 
Drazin inverse of $a$, denoted it by $b=a^D$ (see~\cite{C2,J}). We now prove the main result of this section.

\begin{thm} Let $A,D$ and $BC$ have group inverses. If $AB=0, BD=0, B(CB)^{\pi}=0, C(BC)^{\pi}=0$ and $DC=\lambda CA$, then $M\in M_2(\mathcal{A})^{\#}$ and
$$M^{\#}=\left(
\begin{array}{cc}
2A^{\#}&0\\
-D^{\#}CA^{\#}&D^{\#}+D^{\#}(CB)^{\pi}
\end{array}
\right).$$\end{thm}
\begin{proof} Write $M=P+Q$, where $$P=\left(
  \begin{array}{cc}
    A & 0 \\
    0 & D
  \end{array}
\right), Q=\left(
  \begin{array}{cc}
   0 & B \\
   C & 0
  \end{array}
\right).$$ Then $$P^{\#}=\left(
  \begin{array}{cc}
    A^{\#} & 0 \\
    0 & D^{\#}
  \end{array}
\right), P^{\pi}=\left(
  \begin{array}{cc}
    A^{\pi} & 0 \\
    0 & D^{\pi}
  \end{array}
\right).$$ Since $BC$ has group inverse, we have $$(Q^2)^D=\left(
  \begin{array}{cc}
   BC & 0 \\
   0 & CB
  \end{array}
\right)^D=\left(
  \begin{array}{cc}
   (BC)^{\#} & 0 \\
   0 & (CB)^{\#}
  \end{array}
\right).$$ In view of~|cite[Theorem 2.1]{J}, $Q$ has Drazin inverse and $$Q^D=Q(Q^2)^D=\left(
  \begin{array}{cc}
  0& B(CB)^{\#} \\
  C(BC)^{\#}&0
  \end{array}
\right).$$ Moreover, we have $$Q^{\pi}=\left(
  \begin{array}{cc}
  (BC)^{\pi}&0\\
  0&(CB)^{\pi}
  \end{array}
\right).$$
Hence, $$QQ^D=Q^DQ, Q^D=Q^DQQ^D.$$ It is easy to verify that
$$\begin{array}{lll}
QQ^{\pi}&=&\left(
  \begin{array}{cc}
   0 & B \\
   C & 0
  \end{array}
\right)\left(
  \begin{array}{cc}
   (BC)^{\pi} & 0 \\
  0 & (CB)^{\pi}
  \end{array}
\right)\\
&=&\left(
  \begin{array}{cc}
  0&B(CB)^{\pi}\\
  C(BC)^{\pi} & 0
  \end{array}
\right)\\
&=&0;
\end{array}$$ hence, $Q=Q^2Q^D$. This implies that $Q$ has group inverse and $$Q^{\#}=\left(
  \begin{array}{cc}
  0& B(CB)^{\#} \\
  C(BC)^{\#}&0
  \end{array}
\right).$$
Therefore we get
$$P^{\#}Q^{\pi}=
\left(
  \begin{array}{cc}
  A^{\#}(BC)^{\pi}&0\\
  0&D^{\#}(CB)^{\pi}
  \end{array}
\right).$$
We easily check that
$$PQ=\left(
  \begin{array}{cc}
   0 & 0 \\
    DC & 0
  \end{array}
\right)=\lambda\left(
  \begin{array}{cc}
  0 & 0 \\
    CA & 0
  \end{array}
\right)=\lambda QP.$$ It follows by Lemma 2.1 that
$PQQ^{\#}=QQ^{\#}P$ and $QPP^{\#}=PP^{\#}Q$.
Then we have
$$\begin{array}{c}
BCA=0,CAA^{\#}=DD^{\#}C,\\
D(CB)(CB)^{\#}=(CB)(CB)^{\#}D.
\end{array}$$
Obviously, we have $$\begin{array}{ll}
&\left(
\begin{array}{cc}
A&AA^{\#}B\\
DD^{\#}C&D
\end{array}
\right)\\
=&\left(
\begin{array}{cc}
A&0\\
0&D
\end{array}
\right)[\left(
\begin{array}{cc}
I&0\\
0&I
\end{array}
\right)+\left(
\begin{array}{cc}
A^{\#}&0\\
0&D^{\#}
\end{array}
\right)\left(
\begin{array}{cc}
0&B\\
C&0
\end{array}
\right)].
\end{array}$$
According to Theorem 2.2,
$M$ has group inverse if and only if $\left(
\begin{array}{cc}
A&0\\
DD^{\#}C&D
\end{array}
\right)$ has group inverse.
Obviously, $D^{\pi}(DD^{\#}C)=0$. In view of Lemma 3.1,
$\left(
\begin{array}{cc}
A&0\\
DD^{\#}C&D
\end{array}
\right)$ has group inverse. Moreover, we have
$$\left(
\begin{array}{cc}
A&0\\
DD^{\#}C&D
\end{array}
\right)^{\#}=\left(
\begin{array}{cc}
A^{\#}&0\\
-D^{\#}CA^{\#}+(D^{\#})^2CA^{\pi}&D^{\#}
\end{array}
\right).$$
Therefore
$$\begin{array}{lll}
M^{\#}&=&\big[P(I+P^{\#}Q)\big]^{\#}+P^{\#}Q^{\pi}\\
&=&\left(
\begin{array}{cc}
A&0\\
DD^{\#}C&D
\end{array}
\right)^{\#}+P^{\#}Q^{\pi}\\
&=&\left(
\begin{array}{cc}
A^{\#}+A^{\#}(BC)^{\pi}&0\\
-D^{\#}CA^{\#}+(D^{\#})^2CA^{\pi}&D^{\#}+D^{\#}(CB)^{\pi}
\end{array}
\right),
\end{array}$$ as asserted.\end{proof}

\begin{cor} Let $A,D$ and $BC$ have group inverses. If $CA=0, DC=0, B(CB)^{\pi}=0, C(BC)^{\pi}=0$ and $AB=\lambda BD$, then $M\in M_2(\mathcal{A})^{\#}$ and
$$M^{\#}=\left(
\begin{array}{cc}
A^{\#}+A^{\#}(BC)^{\pi}&-A^{\#}BD^{\#}\\
0&2D^{\#}
\end{array}
\right).$$\end{cor}
\begin{proof} Since $BC$ has group inverse, it has Drazin inverse. According to Cline's formula (~\cite[Theorem 2.3]{J}), $CB$ has Drazin inverse.
Hence, $(CB)(CB)^D=(CB)^D(CB), (CB)^D=(CB)^D(CB)(CB)^D$. Since $B(CB)^{\pi}=0$, we have $CB=(CB)(CB)^D(CB)$. Therefore $CB$ has group invrse.
Applying Theorem 3.2, we prove that $\left(
\begin{array}{cc}
D&C\\
B&A
\end{array}
\right)$ has group inverse. Clearly,
$$M=\left(
\begin{array}{cc}
0&I\\
I&0
\end{array}
\right)\left(
\begin{array}{cc}
D&C\\
B&A
\end{array}
\right)\left(
\begin{array}{cc}
0&I\\
I&0
\end{array}
\right),$$ and so $M$ has group inverse. Moreover,
$\left(
\begin{array}{cc}
D&C\\
B&A
\end{array}
\right)$ has group inverse. Clearly,
$$M^{\#}=\left(
\begin{array}{cc}
0&I\\
I&0
\end{array}
\right)\left(
\begin{array}{cc}
D&C\\
B&A
\end{array}
\right)^{\#}\left(
\begin{array}{cc}
0&I\\
I&0
\end{array}
\right),$$ as required.\end{proof}

\begin{cor} Let $A\in {\Bbb C}^{m\times m},D\in {\Bbb C}^{n\times n}$ have group inverses.
If $AB=0, BD=0, B(CB)^{\pi}=0, C(BC)^{\pi}=0$ and $DC=\lambda CA$, then $M\in {\Bbb C}^{(m+n)\times (m+n)}$ has group inverse and
$$M^{\#}=\left(
\begin{array}{cc}
A^{\#}+A^{\#}(BC)^{\pi}&0\\
-D^{\#}CA^{\#}+(D^{\#})^2CA^{\pi}&D^{\#}+D^{\#}(CB)^{\pi}
\end{array}
\right).$$\end{cor}\begin{proof} Since $C(BC)^{\pi}=0$, we have $BC=(BC)^2(BC)^D$' hence, $BC$ has group inverse. This completes the proof by Theorem 3.2.\end{proof}

We are now ready to prove:

\begin{thm} Let $A,D$ and $BC$ have group inverses. If $CA=0, DC=0, (CB)^{\pi}C=0, (BC)^{\pi}B=0$ and $AB=\lambda BD,$ then $M\in M_2(\mathcal{A})^{\#}$ and
$$M^{\#}=\left(
\begin{array}{cc}
A^{\#}+(BC)^{\pi}A^{\#}&-A^{\#}BD^{\#}+A^{\pi}B(D^{\#})^2\\
0&D^{\#}+(CB)^{\pi}D^{\#}
\end{array}
\right).$$\end{thm}
\begin{proof} By hypothesis, $A^T,D^T$ and $C^TB^T$ have group inverses. Moreover, we check that
$$\begin{array}{c}
D^TB^T=\lambda B^TA^T, A^TC^T=0, C^TD^T=0,\\
C^T(B^TC^T)^{\pi}=0, B^T(C^TB^T)^{\pi}=0.
\end{array}$$
Applying Theorem 3.2 to the transpose $M^T$ of $M$, where $$M^T=\left(
  \begin{array}{cc}
    A^T & C^T \\
    B^T & D^T
  \end{array}
\right).$$ Then $M^T$ has group inverse, and so has $M$. Moreover, we derive
$${\small\begin{array}{l}
(M^T)^{\#}=\\
\left(
\begin{array}{cc}
(A^T)^{\#}+(A^T)^{\#}(C^TB^T)^{\pi}&-(D^T)^{\#}B^T(A^T)^{\#}+[(D^T)^{\#}]^2B^T(A^T)^{\pi}\\
0&(D^T)^{\#}+(D^T)^{\#}(B^TC^T)^{\pi}
\end{array}
\right).
\end{array}}$$ This completes the proof.\end{proof}

As a consequence of the above, we derive

\begin{cor} Let $A,D$ and $BC$ have group inverses. If $AB=0, BD=0 (BC)^{\pi}B=0, (CB)^{\pi}C=0$ and $DC=\lambda CA,$ then $M\in M_2(\mathcal{A})^{\#}$ and
$$M^{\#}=\left(
\begin{array}{cc}
A^{\#}+(BC)^{\pi}A^{\#}&0\\
-D^{\#}CA^{\#}+D^{\pi}C(A^{\#})^2&D^{\#}+(CB)^{\pi}D^{\#}
\end{array}
\right).$$\end{cor}
\begin{proof} Applying Theorem 3.5, $\left(
\begin{array}{cc}
D&C\\
B&A
\end{array}
\right)$ has group inverse. As in the proof of Corollary 3.3, we obtain the result.\end{proof}

\begin{cor} Let $A\in {\Bbb C}^{m\times m},D\in {\Bbb C}^{n\times n}$ have group inverses. If $CA=0, DC=0, (CB)^{\pi}C=0, (BC)^{\pi}B=0$ and $AB=\lambda BD,$ then
$M\in {\Bbb C}^{(m+n)\times (m+n)}$ has group inverse and
$$M^{\#}=\left(
\begin{array}{cc}
A^{\#}+(BC)^{\pi}A^{\#}&-A^{\#}BD^{\#}+A^{\pi}B(D^{\#})^2\\
0&D^{\#}+(CB)^{\pi}D^{\#}
\end{array}
\right).$$\end{cor}

\end{document}